\documentclass[11pt]{article}

\usepackage{amsthm}
\usepackage{amsmath}
\usepackage{amsfonts, epsfig, amsmath, amssymb, color, amscd}
\usepackage[cmtip,arrow]{xy}
\usepackage{pb-diagram, pb-xy}
\usepackage{amssymb,epsfig}
\usepackage{color}
\usepackage[cmtip,arrow]{xy}
\usepackage{pb-diagram, pb-xy}
\usepackage{amssymb,epsfig,amsfonts}

\usepackage[linktocpage=true,plainpages=false,pdfpagelabels=false]{hyperref}

\usepackage{tikz}
\usepackage{eufrak}

\newfont{\sdbl}{msbm9}
\newfont{\dbl}{msbm10 at 12pt}
\theoremstyle{definition}

\newcommand{\dn}{{\mbox{\dbl N}}}

\newcommand{\dc}{{\mbox{\dbl C}}}
\newcommand{\dq}{{\mbox{\dbl Q}}}

\newcommand{\ord}{\mathop{\rm ord}\nolimits}

\newcommand{\End}{\mathop {\rm End}}

\newcommand{\Aut}{\mathop {\rm Aut}}

\newcommand{\rk}{\mathop {\rm rk}}

\newcommand{\Projj}{\mathop {\rm Proj}}

\makeatother
\newtheorem{Def}{Definition}[section]
\newtheorem{rem}{Remark}[section]

\theoremstyle{plain}

\newtheorem{theorem}{Theorem}[section]
\newtheorem{lemma}{Lemma}[section]

\numberwithin{equation}{section}

\begin{document}
\title{On some questions around Berest's conjecture}
\author{Junhu Guo \and  A.B. Zheglov}
\date{}
\maketitle

\begin{abstract}
Let $K$ be a field of characteristic zero, let $A_1=K[x][\partial ]$ be the first Weyl algebra. 
In this paper we prove the following two results.

Assume there exists a non-zero polynomial $f(X,Y)\in K[X,Y]$, which has a non-trivial solution $(P,Q)\in A_{1}^{2}$ with $[P,Q]=0$, and the number of orbits under the group action of $Aut(A_1)$ on  solutions of $f$ in $A_{1}^{2}$ is finite. Then the Dixmier conjecture holds, i.e $\forall \varphi\in End(A_{1})-\{0\}$, $\varphi$ is an automorphism.

Assume $\varphi$ is an endomorphism of monomial type (in particular, it is not an automorphism, see theorem \ref{T:fixed}). Then it has no non-trivial fixed point, i.e.   there are no $P\in A_1$, $P\notin K$, s.t. $\varphi (P)=P$. 
\end{abstract}

\tableofcontents

\section{Introduction}

Consider an irreducible polynomial $f(X,Y)\in \dc [X,Y]$.  
Assume there exists a pair of operators $(P,Q)\in A_1$, where $A_1= \dc [x][\partial ]$ is the first Weyl algebra, such that $f(P,Q)=0$ and $\ord (P)>0$ or $\ord (Q)>0$, 
where $\ord (P)=k$ if $P=\sum_{i=0}^k a_i \partial^i$ with $a_i\neq 0$
 (we'll call such a pair {\it a non-trivial solution} of the equation $f=0$). 

The group of automorphisms  of the first Weyl algebra $A_1$ acts on the set of solutions of the equation $f(X,Y)=0$,  i.e. if $P,Q\in A_1$ satisfy the equation and $\varphi\in Aut(A_1)$, then $\varphi (P),\varphi (Q)$ also satisfy the equation, for this reason we define two pair $(P,Q)$ ,$(P',Q')$ are in the same orbit iff there exists automorphism $\varphi$, such that $\varphi(P,Q)=(P',Q')$. A natural and important problem is to describe the orbit space of the group action of $Aut(A_1)$ in the set of solutions.

Y. Berest,  cf. \cite{MZh},  proposed the following interesting conjecture:

{\it If the Riemann surface corresponding to the equation $f=0$  has genus $g=1$ then the orbit space is infinite, and if $g>1$ then there are only finite number of orbits.}

This conjecture was first studied in the work \cite{MZh}, where its connection with the well known open conjecture of Dixmier was announced. Recall that the Dixmier conjecture (formulated in his seminal work \cite{Dixmier}) claims that any non-zero endomorphism of $A_{1}$ is actually an automorphism. Analogous conjectures exist for any Weyl algebras and are stably equivalent to the Jacobian conjectures, cf. \cite{Ts1}, \cite{Ts2}, \cite{BK}, and are still open in any dimension. Although the results of the work \cite{MZh} showed that the Berest conjecture is not true as it was stated, it is still interesting, as it is not clear whether there exist examples of Riemann surfaces (or algebraic curves defined by one equation, see the discussion below)  with finite number of orbits.

The aim of this paper is to clarify the connection of the last question with the Dixmier conjecture. First, we can formulate a more general question.

Let $K$ be a field of characteristic zero. From now on we will denote by $A_1$ the first Weyl algebra over $K$: $A_1= K [x][\partial ]$\footnote{Although elements in $A_{1}$ can be written in  many kinds of forms, we'll  write them in the form  $P=\sum a_{ij}x^{i}\partial^{j}$.

In Dixmier's paper this algebra was defined as  $A_{1}=K[p,q]$ with the relation $[p,q]=1$, and he put $p$ on the left of $q$. When comparing, we should take $\partial$ instead of $p$, and $x$ instead of $q$. Though Dixmier used another form of elements (he put $p$ on the left of $q$), the difference is not essential for applications of his results in our paper.}. Consider a non-zero irreducible polynomial 
$$
f(X,Y)=\sum_{i,j=0}^n \alpha_{ij}X^iY^j = 0, \quad \alpha_{ij}\in K.
$$
The old result by Burchnall and Chaundy \cite{BC} says that any two commuting differential operators $P,Q\in D:=K[[x]][\partial ]$ are algebraically dependent, more precisely, there exists an irreducible polynomial $f(X,Y)$ of {\it special form}: $f(X,Y) = \alpha X^m\pm Y^n+\ldots $ (i.e. the weighted degree of $f$ is $mn$, and $\ldots$ mean terms of lower weighted degree), such that $f(P,Q)=0$ (cf. \cite[Lemma 5.23]{Zheglov_book}). Vice versa, if $P,Q\in D$ is a solution of such polynomial $f(X,Y)$\footnote{A solution of the equation $f(X,Y)=0$  is a pair $(P,Q)\in D$ such that $\sum_{i,j=0}^n \alpha_{ij}P^iQ^j = 0$.}, then  $[P,Q]=0$.  Thus, the question of Berest appears to be closely connected with the well known theory of commuting ordinary differential operators (for a review of the modern state of this theory see e.g. the book \cite{Zheglov_book} and references therein)\footnote{A natural question whether $f(P,Q)=0 \Rightarrow [P,Q]=0$ for generic polynomial $f$ appears to be much more difficult, and we are going to return to this question in subsequent paper.}. 

Consider now the commutative ring $B=K[P,Q]\subset A_1$, where $(P,Q)$ is a solution of the equation $f(X,Y)=0$. It is well known that $B$ is an integral domain, i.e. it defines an affine plain algebraic curve. This curve admits a natural compactification. Namely, 
recall that in $A_1$ there is a natural order filtration 
$$
K[x]=A_1^0\subset A_1^1\subset A_1^2 \subset \ldots , \quad A_1^n=\{P\in A_1\, \ord (P)\le n\}.
$$
 This filtration induces the filtration on $B$: $B^n=B\cap A_1^n$. The compactification (the projective spectral curve) is defined as  $C:=\Projj \tilde{B}$, where $\tilde{B}=\oplus_{i=0}^{\infty} B^n$ is the Rees ring. Denote by $p_a(C)$ the arithmetic genus $C$. The Berest conjecture now can be improved as a question what is the orbit space of  solutions of a given irreducible equation $f(X,Y)=0$ depending on the genus $p_a(C)$. 

Let's note that the question whether there exists a non-trivial solution of a given equation is rather non-trivial. First non-trivial examples appeared in the same paper of Dixmier \cite{Dixmier}:
$$
 L_{4}=\left(\partial^2-x^3-\alpha\right)^2-2x,\quad 
 L_{6}=\left(\partial^2-x^3-\alpha\right)^3-
 \frac{3}{2}\left(x\left(\partial^2-x^3-\alpha\right)+
 \left(\partial^2-x^3-\alpha\right)x\right),
$$
$$f(L_4, L_6):\quad  L_6^2=L_4^3-\alpha \quad p_a(C)=1.$$

Further examples were invented in connection with one question of I.M. Gelfand, who asked whether there are commuting operators from the first Weyl algebra such that their spectral curve has any given genus and their rank is a given one, where the rank of any given commutative ring $B\subset A_1$ is defined as the number $\rk B= GCD\{\ord (Q), Q\in B\}$. 

 An answer to this question was given only recently, after works of A. Mironov \cite{M1}, \cite{M2}, who constructed a series of examples of rank two and three operators corresponding to a hyperelliptic curve of any given genus, and works of O. Mokhov \cite{Mokh1}, \cite{Mokh2}, who found a way to produce commuting operators of any rank $>1$ with the help of Mironov's examples. At last, in a  paper \cite{ML} L. Makar-Limanov proved that any ring $B$ of rank one can be embedded into the ring $K[z]$, i.e. the spectral curve of any such ring is rational (and if one considers a non-trivial ring $B$, this curve is singular). Interested reader will also benefit by looking at recent works \cite{Vardan34}, \cite{Previato2019} related to this question.
 
All these results show that there are many non-trivial solutions in $A_1$, i.e. the question of Berest is meaningful. It was studied since the work \cite{MZh} in several papers: in \cite{MZh} and \cite{SMZh} it was shown that the number of orbits is infinite for spectral curves of genus one (more precisely, even the number of orbits of pairs $(L_4, L_6)$ with $\ord (L_4)=4$, $\ord (L_6)=6$ is infinite) and there are examples of  hyperelliptic curves of any genus with infinite number of orbits. In the paper \cite{MirDav} the same result was established for almost all hyperelliptic curves of genus 2 (the field $K=\dc$ in all these works). An interesting result for $K=\dq$ was obtained in a recent paper \cite{Gundareva} by A. Gundareva: for an elliptic curve over $\dq$ the orbits of pairs $(L_4, L_6)$, where 
$L_{4}=(\partial^2+V(x))^2+W(x),$ and $W(x)$ is a polynomial of order $2$, are parametrized by the rational points of this curve. 

The first result of our paper is related to the Berest conjecture as follows: if there exists a non-zero polynomial such that the number of orbits of its non-trivial solutions in $A_1$ is finite and non-zero, then the Dixmier conjecture is true, see theorem \ref{T:main}. 

In view of this statement it is reasonable to continue the study of the Berest conjecture. Since commuting differential operators admit the well known effective classification in terms of algebro-geometric spectral data, it would be interesting to understand the action of the automorphism group $\Aut (A_1)$ on these data. Let's mention that even in the simplest most studied case of genus one spectral curves (see the works \cite{KN}, \cite{G}, \cite{Mokh}, \cite{Grun}, \cite{PZ}, \cite{PW}, \cite{BZ}) this question has not yet been considered, although the methods of the last cited papers allow us to hope that it is possible. 

The second result is also motivated by the Berest and Dixmier conjectures. It is well known, see e.g. \cite[Prop. 6.9]{GGV}, \cite[Corol. 2.5]{Joseph}, that if there exists a counter-example to the Dixmier conjecture (i.e. a non-zero endomorphism which is not an automorphism), then after a composition with some automorphism it becomes of subrectangular type, see  definition \ref{D:subrectangular}. The second theorem says that any non-zero endomorphism of monomial type (it is a weaker property, see definition \ref{D:monomial_type}) has no fixed points, see theorem \ref{T:fixed}. It is related to the first theorem: if  there exists a non-zero polynomial with finite solution orbit, then there should exist an endomorphism of monomial type with a fixed point, see section \ref{S:second}.  

The paper is organized as follows. In section 2 we introduce notation and recall necessary facts from the Dixmier paper \cite{Dixmier}. In section 3 we prove the first theorem. In section 4 we prove the second theorem.

{\bf Acknowledgements.} The work was partially supported by the National Key R and D Program of China (Grant No. 2020YFE0204200), by the School of Mathematical Sciences, Peking University and Sino-Russian Mathematics Center as well as by the Moscow Center of Fundamental and applied mathematics at Lomonosov Moscow State  University.

We are grateful to A.Ya. Kanel-Belov for useful remarks and stimulating questions. 

We are also grateful to the anonymous referee  for his useful hints and for pointing out the paper \cite{GGV}, which allowed us to significantly shorten the proofs of our results.

\section{Preliminaries}

Suppose $K$ is a field of characteristic zero, with algebraic closure $\bar{K}$. Denote by $\bar{A}_1=A_1\otimes_K\bar{K}$.

Recall that the Newton Polygon of a polynomial $f(x,y)=\sum c_{ij}x^{i}y^{j}\in K[x,y]$ is the convex hull of all the points $(i,j)$ with respect to those $c_{ij}\neq 0$.

Following \cite{Dixmier}, we define $E(f)=\{(i,j)|c_{ij}\neq 0 \}$ as the points set of the Newton Polygon. Suppose $\sigma, \rho$ are two real numbers, we define the weight degree $v_{\sigma,\rho}(f)$ and the top points set $E(f,\sigma,\rho)$ as follows:
\begin{Def}
	$v_{\sigma,\rho}(f)=sup\{\sigma i+\rho j|(i,j)\in E(f)\}, \\E(f,\sigma,\rho)=\{(i,j)|v_{\sigma,\rho}(f)=\sigma i +\rho j \}$. 
\end{Def}

If $E(f)=E(f,\sigma,\rho)$, then we call such $f$  {\it $(\sigma,\rho)$-homogeneous}. 
\begin{Def}
For any polynomial $f$, we call  $f_{\sigma , \rho}=f_{\sigma , \rho} (f):= \sum_{(i,j)\in E(f,\sigma,\rho)} c_{ij}x^{i}y^{j}$ the {\it $(\sigma, \rho)$-homogeneous polynomial associated to} $f$. 

The line $l$ through the points of the Newton polygon of $f_{\sigma , \rho}$ is called the {\it $(\sigma, \rho)$-top line} of the Newton polygon.
\end{Def}

We define the Newton Polygon of an operator $P\in A_1$, $P= \sum a_{ij}x^{i}\partial^{j}$ as the Newton polygon of the polynomial $f= \sum a_{ij}x^{i}y^{j}$, and we'll use the same notations $v_{\sigma,\rho}(P)$, $f_{\sigma ,\rho}(P)$, $E(P)$ for operators.

Set $\ord_x(P):=v_{1,0}(P)$. Note that $\ord (P)=v_{0,1}(P)$. For any differential operator (not necessarily from $A_1$) define $HT(P):= a_n$, if $P=\sum_{i=0}^na_i\partial^i$ with $a_n\neq 0$. We'll say $P$ is {\it monic} if $HT(P)=1$, and we'll say $P$ is {\it formally elliptic} if $HT(P)\in K$.

For any two polynomials $f,g$ we define the standard Poisson bracket $\{f,g\}=\frac{\partial f}{\partial x}\frac{\partial g}{\partial y}-\frac{\partial f}{\partial y}\frac{\partial g}{\partial x}$.  We denote $\{ad f\}g :=\{f,g\}$, and for any integer $n,\{ad f\}^{n}g:=\{f,\{f,\cdots g\}\}$. As usual, $(ad (P))Q:=[P,Q]$, where $P,Q\in A_1$. 

The {\it tame automorphisms} of $A_{1}$ are defined as follows: for any integer $n$, and $\lambda\in K$, define 
\begin{equation}
\label{E:tame}
\Phi_{n,\lambda}:
\begin{cases}
\partial  \rightarrow & \partial
\\x   \rightarrow & x+\lambda \partial^{n}
\end{cases},  \quad 
\Phi'_{n,\lambda}:
\begin{cases}
\partial  \rightarrow & \partial+\lambda x^{n}
\\x   \rightarrow & x
\end{cases},  
\quad 
\Phi_{a,b,c,d}:
\begin{cases}
\partial  \rightarrow & a\partial+bx
\\x   \rightarrow & c\partial +dx
\end{cases}  
\end{equation}
with $ad-bc=1$. According to \cite[Th. 8.10]{Dixmier} the group $Aut (A_1)$ is generated by all $\Phi_{n,\lambda}, \Phi'_{n,\lambda}$ and $\Phi_{a,b,c,d}$.



Below we'll need to introduce two definitions related to one important result from the Dixmier paper. For the convenience of the reader we recall it here:

\begin{theorem}{(\cite[Lemma 2.4, 2.7]{Dixmier})}
\label{T:Dixmier2.7}
	Suppose $P,Q$ are two differential operators from $A_{1} $, $\sigma, \rho$ are real numbers, with $\sigma +\rho >0, v=v_{\sigma.\rho}(P) $ and $ w=w_{\sigma.\rho}(Q) $. If now $f_{1}$, $g_{1}$ are the $(\sigma,\rho)$-homogeneous polynomials  associated to $P, Q$, then 
	\begin{enumerate}
	\item
	$\exists T,U\in A_{1}$ such that
	\begin{enumerate}
	\item
	$[P,Q]=T+U$
	\item
	$v_{\sigma.\rho}(U)<v+w-\sigma-\rho$
	\item
	$E(T)=E(T,\sigma,\rho)$ and $v_{\sigma.\rho}(T)=v+w-\sigma-\rho$ if $T\neq 0$.
	\end{enumerate}
	\item
	The following are equivalent
	\begin{enumerate}
	\item
	T=0
	\item
	$\{f_{1},g_{1}   \}=0$;
	
	If $v,w$ are integers, then these two conditions are equivalent to
	\item
      $g_{1}^{v}$ is proportional to $f_{1}^{w}$ 
      \end{enumerate}
  \item Suppose $T\neq 0$, then the polynomial $(\sigma,\rho)$ associated to $[P,Q]$ is $\{f_{1},g_{1}\}$.
  \item We have $f_{\sigma ,\rho} (PQ)=f_1 g_1$, $v_{\sigma ,\rho}(PQ)=v+w$. 
	\end{enumerate}
\end{theorem}
\begin{Def}
\label{D:almost_commute}
In the notation of theorem above, if $[P,Q]=T+U$ with $T=0$ (equivalently, $\{f_{\sigma, \rho}(P), f_{\sigma ,\rho}(Q)\}=0$), then we'll say that $P,Q$  {\it almost commute}. 

If $f=f_{\sigma,\rho}(P), g=g_{\sigma,\rho}(Q)$  have some powers proportional to each other, we'll say $f,g$  almost commute as well. 
\end{Def}

\begin{Def}
\label{D:monomial_type}
Suppose $(\sigma,\rho)$ are real non-negative numbers. $P\in A_1$ is called of {\it monomial type}, if the $(\sigma,\rho)$-homogeneous polynomial of $P$ is a monomial.
	
	Suppose $\varphi$ is an endomorphism, and the $(\sigma,\rho)$-homogeneous polynomials of $\varphi(\partial),\varphi(x)$  are monomials (which almost commute by theorem above). Then we call $\varphi$ of {\it monomial type associated with} $(\sigma,\rho)$.
\end{Def}

\section{From the Berest conjecture to the Dixmier conjecture}

In this section we'll prove our first result: 

\begin{theorem} 
	\label{T:main}
	Assume there exists a non-zero polynomial $f(X,Y)\in K[X,Y]$, which has a non-trivial commuting solution pair $(P,Q)\in A_{1}^{2}$, and the number of orbits under the group action of $Aut(A_1)$ on  solutions of $f$ in $A_{1}^{2}$ is finite. Then the Dixmier conjecture holds, i.e $\forall \varphi\in End(A_{1})-\{0\}$, $\varphi$ is an automorphism.
\end{theorem}

The proof will follow from some preliminary lemmas. First recall one definition, cf. \cite[P. 620]{GGV}. 

\begin{Def}
\label{D:subrectangular}
	The Newton Polygon $N$ is called of subrectangular type, if $N$ is contained in the rectangle constructed by $\{(0,0),(0,k),(l,0),(l,k)\}$, and $(l,k)$ is the vertex of $N$ with $l,k \ge 1$. 
	
	Let $\varphi$ be an endomorphism of $A_1$. If the Newton Polygons of $\varphi(\partial),\varphi(x)$ are of subrectangular type, then we call such $\varphi$ of subrectangular type (a specially condition of monomial type).
\end{Def}

It is known (see \cite[Prop. 6.9]{GGV}) that for given $\varphi\in \End (A_1)\backslash \Aut (A_1)$ there is an automorphism $\psi$ such that $\psi\circ\varphi$ is of subrectangular type.  
So, from now on, we can always assume that $\varphi\in \End (A_1)$ is of the subrectangular type. It's easy to see that  for any positive pair $(\sigma,\rho)$, $f_{\sigma,\rho}(\varphi(\partial))$ is a monomial. In contrast, the Newton polygon of an automorphism $\Phi$ has the following properties.

\begin{lemma} \label{L:8}
	Any automorphism of $A_{1}$ can't be of subrectangular type. More precisely: if $\Phi \in Aut(A_{1})$, with $\Phi (\partial )=\sum a_{ij} x^i\partial^j$, and 	$n=sup\{i|\quad \exists j, a_{ij}\neq 0\}$, $m=sup\{j|\quad \exists i, a_{ij}\neq 0  \}$, then 
\begin{enumerate}
		\item $(0,n),(m,0)\in E(\Phi(\partial))$;
		\item $\forall i,j>0,(i,n),(m,j)\notin E(\Phi(\partial)) $; in particular, $f_{n,m}(\Phi(\partial),\Phi(x))$ are not monomials;
		\item one of the rate $\frac{m}{n}$ or $\frac{n}{m}$ must be an integer.
\end{enumerate}
\end{lemma}

The proof is contained e.g. in \cite[L. 8.7]{Dixmier} and \cite[Prop. 6.3]{GGV}. 

\begin{Def}
Suppose $\varphi$ is an endomorphism of subrectangular type. Let $(\sigma,\rho)$ be any  two positive integer numbers.

Assume $f_{\sigma,\rho}(\varphi(\partial))=c_1 x^{l}y^{k}$, $f_{\sigma,\rho}(\varphi(x))=c_2 x^{l'}y^{k'}$, where $l,k,l',k'$ are non-zero integers, $c_1, c_2\in K$. By theorem \ref{T:Dixmier2.7}, $f_{\sigma,\rho}(\varphi(\partial))$ is proportional with $f_{\sigma,\rho}(\varphi(x))$. Define $\epsilon(\varphi)\in \mathbf{Q}$ as
$$
(l',k')=\epsilon(\varphi)(l,k)
$$

We'll call this  $\epsilon(\varphi)$ as the rate of $\varphi(\partial)$ with $\varphi(x)$. Note that $\epsilon(\varphi)$ does not depend on $(\sigma,\rho)$. 
\end{Def}

Before we continue our discussion, we need the following  Lemma (we were unable to find an appropriate reference for it). 

\begin{lemma}\label{L:9}
	Suppose $\varphi$ is an endomorphism of subrectangular type. 
	\begin{enumerate}
\item	For any $q\in\dn$ the endomorphism $\varphi^q$ is of subrectangular type with $\epsilon(\varphi^{q})=\epsilon(\varphi)$.
	
\item	Assume $P\in A_{1}$, $P\notin K$ is  such that $f_{\epsilon,1}(P)$ is a monomial. Then  for any $\sigma,\rho >0$ the polynomial $f_{\sigma,\rho}(\varphi(P))$ is a monomial with 
\begin{equation} \label{E:vi}
	v_{\sigma,\rho}(\varphi(P))=(\rho k+\sigma l)v_{\epsilon,1}(P), 
\end{equation}
where $f_{\sigma,\rho}(\varphi(\partial))=c_1x^{l}y^{k}$, $f_{\sigma,\rho}(\varphi(x))=c_2x^{l\epsilon}y^{k\epsilon}$.
 \end{enumerate}
\end{lemma}
 
 \begin{proof} 1. Note that for any monomial $x^{u}\partial^{v}$ we have $\varphi(x^{u}\partial^{v})=\varphi(x)^{u}\varphi(\partial)^{v}$ and by theorem \ref{T:Dixmier2.7}, item 4 we have therefore 
 $$
 f_{\sigma ,\rho}(\varphi(x^{u}\partial^{v}))=cx^{l(u\epsilon+v)}y^{k(u\epsilon +v)}
 $$
for any positive $\sigma ,\rho$  and some $0\neq c\in K$. In particular, we can choose $\sigma ,\rho$ in such a way that the corresponding $(\sigma ,\rho)$-line $\sigma x+\rho y=\theta$ is almost flat or almost vertical. More precisely, we need to choose such $\sigma ,\rho$ that the lines through the vertex $(l,k)$ and through the vertex $(l\epsilon ,k\epsilon)$ intersect the vertical axe at the point $(0,a)$ and correspondingly at the point $(0,a\epsilon)$ with $a<k+1$ and $a\epsilon <k\epsilon +1 $ (and choose such $\sigma ,\rho$ that this lines intersect the flat axe at the point $(b,0)$ and correspondingly at the point $(b\epsilon,0)$ with $b<l+1$ and $b\epsilon <l\epsilon +1$). With this choice it just suffices to prove that the monomials $f_{\sigma ,\rho}(\varphi(x^{u}\partial^{v}))$, where $(u,v)$ are the vertices inside the rectangle $\{(0,0),(0,k),(l,0),(l,k)\}$,  belong to the rectangle $\{(0,0),(0,k(l\epsilon +k)),(l(l\epsilon +k),0),(l(l\epsilon +k),k(l\epsilon +k))\}$ (and correspondingly, the monomials $f_{\sigma ,\rho}(\varphi(x^{u}\partial^{v}))$, where $(u,v)$ are the vertices inside the rectangle $\{(0,0),(0,k\epsilon),(l\epsilon,0),(l\epsilon,k\epsilon)\}$,  belong to the rectangle $\{(0,0),(0,\epsilon k(l\epsilon +k)),(\epsilon l(l\epsilon +k),0),(\epsilon l(l\epsilon +k), \epsilon k(l\epsilon +k))\}$). Indeed, if it is true, then  all other monomials of  $\varphi(x^{u}\partial^{v})$ will belong to the same rectangles $\{(0,0),(0,k(l\epsilon +k)),(l(l\epsilon +k),0),(l(l\epsilon +k),k(l\epsilon +k))\}$, $\{(0,0),(0,\epsilon k(l\epsilon +k)),(\epsilon l(l\epsilon +k),0),(\epsilon l(l\epsilon +k), \epsilon k(l\epsilon +k))\}$, i.e. $\varphi$ will be of subrectangular type.

But this is easy: $(u,v)$ is a vertex inside the rectangle $\{(0,0),(0,k),(l,0),(l,k)\}$ iff $u\le l$ and $v\le k$. Clearly, if these inequalities hold, then $l(u\epsilon+v)\le l(l\epsilon +k)$ and $k(u\epsilon +v)\le k(l\epsilon +k)$, i.e. $f_{\sigma ,\rho}(\varphi(x^{u}\partial^{v}))$ belongs to the rectangle $\{(0,0),(0,k(l\epsilon +k)),(l(l\epsilon +k),0),(l(l\epsilon +k),k(l\epsilon +k))\}$. 

Thus, $\varphi^2$ is of subrectangular type, and the same arguments show that $\varphi^q$ is of subrectangular type for any $q\in \dn$. A simple calculation now shows that: for any  $q\in \dn$,
\begin{equation}
\label{E:cases}
\begin{cases}
	f_{\sigma ,\rho}(\varphi^{q}(\partial))=\tilde{c}_1x^{l(k+\epsilon l)^{q-1}}y^{k(k+\epsilon l)^{q-1}}
	\\f_{\sigma ,\rho}(\varphi^{q}(x))=\tilde{c}_2x^{l\epsilon(k+\epsilon l)^{q-1}}y^{k\epsilon(k+\epsilon l)^{q-1}}
\end{cases}
\end{equation}
so that $\epsilon(\varphi^{q})=\epsilon(\varphi)$, too.

2. Suppose $P=cx^{u}\partial^{v} +g(x,\partial)$, with $f_{\epsilon,1}(P)=cx^{u}y^{v}$. Then $\varphi(P)=\varphi(cx^{u}\partial^{v})+\varphi(g)$. Note that for any monomial $x^{u'}\partial^{v'}$ we have  
 $$v_{\sigma,\rho}(\varphi(x^{u'}\partial^{v'}))= (\rho k+\sigma l)(u'\epsilon+v').
 $$ 
Since the $(\epsilon,1)$-weights of all monomials in $g$ are less than $(u\epsilon+v )$, we get from this  
 $$v_{\sigma,\rho}(\varphi(P))= (\rho k+\sigma l)v_{\epsilon,1}(P), \quad 
 v_{\sigma,\rho}(\varphi(g))< (\rho k+\sigma l)v_{\epsilon,1}(P),
 $$ 
 and $f_{\sigma,\rho}(\varphi(P))=f_{\sigma,\rho}(\varphi (cx^{u}\partial^{v}))$ is a monomial by theorem \ref{T:Dixmier2.7}, item 4 (because $f_{\sigma,\rho}(\varphi (x), \varphi (\partial ))$ are monomials for any $\sigma, \rho >0$). 
\end{proof}

\begin{lemma}
\label{L:infinte orbit}
	Suppose $P,Q\in A_1$ is a solution of a non-zero polynomial equation $f(X,Y)=0$, suppose $\varphi\in End(A_1)-Aut(A_1)$ is of subrectangular type. Then there exist automorphism $\psi $ of $A_1$ such that for $i\neq j$, $i>0$, $j>0$, $\varphi^i\circ \psi (P,Q)$ doesn't lie in the same orbit with $\varphi^j\circ\psi(P,Q)$. 
\end{lemma}

\begin{proof}
    Suppose as before $f_{\sigma,\rho}(\varphi(\partial))=c_1 x^{l}y^{k}$, $f_{\sigma,\rho}(\varphi(x))=c_2 x^{l'}y^{k'}$ (for any positive $\sigma ,\rho$), where $l,k,l',k'$ are non-zero integers, $c_1, c_2\in K$, $\epsilon=\epsilon(\varphi)$. Now we can choose an integer $N$ big enough such that the automorphism:
    $$ \psi:
    \begin{cases}
    \partial\rightarrow\partial
    \\x\rightarrow x+\partial^N
    \end{cases}
    $$
    satisfies the following properties:
    \begin{enumerate}
    	\item $\psi(P,Q)$ are both formally elliptic.
    	\item $f_{\epsilon, 1}(\varphi \circ\psi (P,Q))$ is a monomial. 
    \end{enumerate}
    This can be done since our $\epsilon$ has been fixed when we choose $N$.
    
    According to Lemma \ref{L:9}, we know for any $i\in \dn$ that the Newton polygons of $\varphi^i\circ \psi(P,Q)$ are of subrectangular type, because $f_{\sigma,\rho}(\varphi^i \circ\psi (P,Q))$ are monomials for any $\sigma, \rho >0$. Now it's enough to show for any $j>i>0$, $\varphi^j\circ \psi(P,Q)$ lies in a different orbit with $\varphi^i\circ \psi(P,Q)$.
    
   Assume the converse. Then, according to the definition of the orbit, we know there exists an automorphism $\Phi\in Aut(A_1)$ such that 
    $$
    \Phi \circ \varphi^i\circ \psi(P,Q)=\varphi^j\circ \psi(P,Q).
    $$
     According to Lemma \ref{L:9} again we know that the Newton polygons of $\varphi^j\circ \psi(P,Q)$, $\varphi^i\circ \psi(P,Q)$ are of subrectangular type. 
    
    According to lemma \ref{L:8} there exist $\sigma, \rho >0$ such that $f_{\sigma,\rho} (\Phi (\partial ), \Phi (x))$ are not monomials. Then if, say, $f_{\sigma,\rho}(\varphi^i \circ\psi (P))=c x^{i'}y^{j'}$, we'll have by theorem \ref{T:Dixmier2.7} item 4 that 
$$
f_{\sigma,\rho} (\Phi (\varphi^i \circ\psi (P)))= f_{\sigma,\rho} (\Phi (c x^{i'}\partial^{j'}))= 
c f_{\sigma,\rho} (\Phi (x))^{i'} f_{\sigma,\rho} (\Phi (\partial ))^{j'} 
$$ 
is not a monomial, but $f_{\sigma,\rho} (\varphi^i \circ\psi (P))$ is a monomial -- a contradiction. 
\end{proof}

\begin{proof} (of theorem \ref{T:main}) Suppose  Dixmier's conjecture is not true, hence there exists an endomorphism $\varphi\in End(A_1)-\{0\}$ such that $\varphi$ is not a automorphism, where  we can asssume $\varphi$ is of subrectangular type. Assume $(P,Q)$ is a solution  of the equation $f=0$. According to Lemma \ref{L:infinte orbit}, we know there exists automorphism $\psi$, such that for any integer $i$, $\varphi^i\circ\psi(P,Q)$ lie in different solution orbits, so that the number of solution orbits is infinite,  a contradiction.
\end{proof}

\section{Fixed points of endomorphisms}
\label{S:second}

Let's start with the following motivating observation. 
Let $\varphi$ be an endomorphism of subrectangular type, and assume it is not an automorphism. Without loss of generality we can assume that neither of $\epsilon (\varphi ),\frac{1}{\epsilon (\varphi )}$ is integer (see \cite[Prop. 6.3]{GGV} and \cite[L.6.2(6)]{GGV}). Assume that $f_{\sigma,\rho} (\varphi (\partial ))=\tilde{c}_1 x^l \partial^k$, $f_{\sigma,\rho} (\varphi (x))=\tilde{c}_2 x^{\epsilon l} \partial^{\epsilon k}$ (here $\sigma , \rho$ are arbitrary positive numbers).

Suppose the conditions of theorem \ref{T:main} holds, choose a non-trivial equation $F=0$ such that the number of solution orbits is finite. Then  choose a non-trivial solution $(P,Q)$ in a solution orbit, choose $(\sigma,\rho)=(p\epsilon (\varphi ),p)$ such that $\sigma, \rho \in \dn$. 

W.l.o.g we can suppose $f_{\sigma,\rho}(P),f_{\sigma,\rho}(Q)$ are monomials: if not, take automorphism $\Phi_{N,1}$ (or $\Phi_{N,1}'$) with $N\gg 0$ an integer big enough, then $(\Phi(P),\Phi(Q))$ is also a non-trivial solution to the equation $F=0$, and $f_{\sigma,\rho}(\Phi(P)),f_{\sigma,\rho}(\Phi(Q))$ are monomials. 

Since the number of orbits of solutions is finite, there exist  $p,q\in \dn$ such that $\varphi^{p+q}(P,Q)$ is on the same orbit with $\varphi^{q}(P,Q)$. According to the definition of the orbit, we know $\exists \psi \in \Aut( A_{1})$, such that 
\begin{equation} \label{E:v}
\varphi^{p}(\varphi^{q}(P,Q))=\psi\circ \varphi ^{q}(P,Q) 
\end{equation}

By lemma \ref{L:9}  we have:  
$$
v_{\sigma,\rho}(\varphi (\partial ,x))=(\rho k+ \sigma l) v_{\epsilon ,1}(\partial ,x) = (k+ \epsilon l) v_{\sigma,\rho} (\partial ,x),
$$
whence 
$$f_{\sigma,\rho}(\varphi^{q}(\partial))=\tilde{c}_1x^{l(k+\epsilon l)^{q-1}}y^{k(k+\epsilon l)^{q-1}},\quad f_{\sigma,\rho}(\varphi^{q}(x))=\tilde{c}_2x^{l\epsilon(k+\epsilon l)^{q-1}}y^{k\epsilon(k+\epsilon l)^{q-1}}$$
are non-trivial monomials for any $q\in\dn$ and  therefore $f_{\sigma,\rho}(\varphi^{q}(P)),f_{\sigma,\rho}(\varphi^{q}(Q))$ are monomials for arbitrary  $q\in\dn$, too.

Observe that the $(\sigma,\rho)$-top line of the Newton polygon of $\psi^{-1}(\partial )$  can only go across one point of $\psi^{-1}(\partial)$, say $(0,n)$ or $(m,0)$, since one of $\frac{n}{m},\frac{m}{n}$ must be an integer (cf. lemma \ref{L:8}) and neither of $\epsilon,\frac{1}{\epsilon}$ is an integer. So the $(\sigma,\rho)$-homogeneous polynomial associated to $\psi^{-1}(\partial)$ is a monomial, hence the $(\sigma,\rho)$-homogeneous polynomial  associated to $\psi^{-1}(x)$ is a monomial too by theorem \ref{T:Dixmier2.7}. Using lemma \ref{L:9} again we get that $f_{\sigma, \rho}(\varphi^{p}\circ\psi^{-1}(x,\partial))$ are monomials for any $p\in\dn$. 

Now let $p,q$ be as in \eqref{E:v}. Put $\tilde{\varphi}=\varphi^{p}\circ \psi^{-1},(\tilde{P},\tilde{Q})=\varphi^{p}(\varphi^{q}(P,Q))=\psi\circ \varphi ^{q}(P,Q)$. Simple calculation shows that $\tilde{\varphi}(\tilde{P},\tilde{Q})=(\tilde{P},\tilde{Q})$. Now take the new $\tilde{\varphi}$ instead of $\varphi$, and $\tilde{P}$ instead of $P$, we get that the new $\tilde{\varphi}$ has a fixed point $\tilde{P}$.

Note that the new $\tilde{\varphi}$ may not be of subrectangular type any more. But it is of monomial type associated to the old $(\sigma,\rho)=(p\epsilon,p)$ and simple observation shows that the vertex in the top line does not lie on the axis, besides, $k,l>1$. The new rate of $\tilde{\varphi}$ may change, too (actually, it must be changed), i.e $\epsilon(\tilde{\varphi})\neq \epsilon (\varphi )$.\footnote{The reason why $\tilde{\epsilon}=\epsilon(\tilde{\varphi})$ must be changed: if not, $f_{\sigma,\rho}(P)$ can't be a monomial any more. This is because $\varphi(P)=P$, so $v_{\sigma,\rho}(\varphi(P))=v_{\sigma,\rho}(P)$. If $f_{\sigma,\rho}(P)$ is a monomial, then $v_{\sigma,\rho}(\varphi(P))=(k+\epsilon l)v_{\sigma,\rho}(P)$ due to lemma \ref{L:9} -- a contradiction, since $P$ is not a constant and therefore $v_{\sigma,\rho}(P)>0$.}
		
Our next aim is to prove that, in fact, an endomorphism of monomial type can not have a fixed point. In particular, this will imply another proof of our first theorem. 

\begin{theorem} 
	\label{T:fixed}
	Suppose $\varphi$ is an endomorphism of $A_{1}$ of monomial type associated to $(\sigma_{0},\rho_{0})$, $ \sigma_{0},\rho_{0} \in \dn$. Suppose $f_{\sigma_{0},\rho_{0}}(\varphi (\partial ))=c x^ly^k$, where $c\in K$, $l,k>0$ and either $l$ or $k$ is greater than 1. 
	
	Then $\varphi$ has  no non-trivial fixed points, i.e. there are no $P\in A_1$, $P\notin K$, s.t. $\varphi (P)=P$. 
\end{theorem}

The proof will be divided into several steps.

\subsection{Step 1}

We start with the following lemma

\begin{lemma}\label{L:10}
	Suppose $\varphi$ is an endomorphism of $A_{1}$ of monomial type associated to $(\sigma_{0},\rho_{0})$, $ \sigma_{0},\rho_{0} \in \dn$. Suppose $f_{\sigma_{0},\rho_{0}}(\varphi (\partial ))=c x^ly^k$, where $c\in K$, $l,k>0$ and either $l$ or $k$ is greater than 1. Suppose that there exists  $P\in A_1$, $P\notin K$, with $\varphi(P)=P$. 
	
	Then there exists an automorphism $\Phi$ such that  $\tilde{\varphi}:=\Phi\circ \varphi\circ \Phi^{-1}, \tilde{P}=\Phi(P)$ satisfy 
	\begin{enumerate}
		\item [(1)]$\tilde{\varphi}$ is of subrectangular type,
		\item[(2)] $\tilde{\varphi}(\tilde{P})=\tilde{P}$, $\tilde{P}\notin K$.
	\end{enumerate}
	
	\begin{proof} Item (2) is obvious,  so we pay our attentions to (1). The proof uses a standard technique from the paper \cite{Dixmier}, cf. \cite{GGV}. Again we give it here, because we were unable to find an appropriate reference. 
		
		\begin{enumerate}
			\item [(I)] Consider first the $(\sigma , \rho )$-top line of the Newton polygon of $\varphi (\partial )$ passing through the vertex $(l,k)$ and a vertex $(\tilde{l}, \tilde{k})$ with $\tilde{k}<k$. By \cite[7.4]{Dixmier} combined with \cite[7.3]{Dixmier} the polynomial $f_{\sigma ,\rho}(\varphi (\partial ))$ is of the form from items (c) or (d) in \cite[7.3]{Dixmier}. Take $\Phi\in \Aut(A_{1})$
			$$\Phi :\begin{cases}
			\partial \rightarrow &\partial
			\\x \rightarrow & x+\mu \partial^{\frac{\sigma}{\rho}}
			\end{cases}$$
			for an appropriate $\mu\in K$. 
			Then put $\tilde{\varphi}=\Phi\circ \varphi\circ \Phi^{-1}, \tilde{P}=\Phi(P)$. Note that $\Phi^{-1}(\partial)=\partial$, so that the Newton Polygon of $\tilde{\varphi}(\partial)$ is totally equal to $\Phi \circ \varphi(\partial)$. We know the Newton Polygon of $\tilde{\varphi}(\partial)$ becomes flatter than before, and we can continue on this procedure until it is strictly flat, i.e. the Newton polygon of $\tilde{\varphi} (\partial )$ lies under the horizontal line going through $(l,k)$. We may denote the result as $\varphi,P$ again.
			\item[(II)] Now we should pay our attention to the Newton Polygons of $\varphi(x)$. Since the Newton polygon of $\varphi(\partial)$ has $(0,1)$-top line passing across $(l,k)$  and $\varphi(\partial)$  almost commutes with $\varphi(x)$, we know by theorem \ref{T:Dixmier2.7} that $\varphi(x)$ has Newton Polygon with the $(0,1)$-top line passing across $(l',k')$  too. More specifically, it is $k'=sup\{j|(i,j)\in E(\varphi(x))\}$ and $l'=sup\{i|(i,k')\in E(\varphi(x))\}$. Note that either $l'$ or $k'$ is greater than 1 (they are related to $k,l$ via multiplication by $\epsilon (\varphi )$).
			
			Consider now  the $(\sigma , \rho )$-top line passing through the vertex $(l',k')$ and a vertex $(\tilde{l'}, \tilde{k'})$ with $\tilde{k'}> k'$. By the same arguments as in (I), in this case $f_{\sigma ,\rho}(\varphi (\partial ))$ is of the form from items (b) or (d) in \cite[7.3]{Dixmier}.
			Now take 
			$$\Phi:\begin{cases}
			\partial \rightarrow &\partial +\mu x^{\frac{\rho}{\sigma}}
			\\x \rightarrow & x
			\end{cases}$$
			for an appropriate $\mu\in K$ 
			and $\tilde{\varphi}=\Phi\circ \varphi\circ \Phi^{-1}, \tilde{P}=\Phi(P)$. 
			
Note that $\Phi^{-1}(x)=x$, so that the Newton Polygon of $\tilde{\varphi}(x)$ is totally equal to $\Phi \circ \varphi(x)$, and the right top line of the Newton Polygon of $\tilde{\varphi}(x)$ becomes more vertical than before. As in (1), we can continue this procedure until the right top line of $\tilde{\varphi}(x)$ becomes strictly vertical, after that we  get $\tilde{\varphi}$ of subrectangular type in view of  theorem \ref{T:Dixmier2.7}.
		\end{enumerate}
	\end{proof}
\end{lemma}

\subsection{Step 2}
\label{S:ooo}

Suppose $\varphi$ is an endomorphism of subrectangular type with a fixed point $P\notin K$ (i.e $\varphi(P)=P$), $\epsilon :=\epsilon (\varphi )$, $(\sigma ,\rho)=(\tilde{p}\epsilon, \tilde{p})\in \dn^2$ and suppose $f_{\sigma ,\rho}(\varphi(\partial))=c_1x^{l}y^{k}$, $f_{\sigma ,\rho}(\varphi(x))=c_2x^{l'}y^{k'}$, with $(l',k')=\epsilon(l,k)$.

Since $\ord_{x}(P)<\infty $, we know $\exists N$, such that $(ad\partial)^{N}P=0$, and since $\varphi$ is an endomorphism, then for any integer $q>0, (ad\varphi^{q}(\partial))^{N}P=0 $. Then due to theorem \ref{T:Dixmier2.7}, we know that $\exists m<N$ such that $(ad\varphi^{q}(\partial))^{m}P$  almost commutes with $\varphi^{q}(\partial)$(then it almost commutes with $\varphi(\partial)$). 

For any $q\in \dn$ we define the number $m(q)$ as {\it the first positive integer} such that $(ad\varphi^{q}(\partial))^{m(q)}P$  almost commutes with $\varphi^q (\partial)$  (it is well-defined as we discuss above, and $m(q)<N$).

In the same way we define the number $n(q)$ to be {\it the first positive integer} such that \\
$f_{\sigma ,\rho}((ad\varphi^{q}(\partial))^{n(q)}P)$ is a monomial. Obviously, it's well defined, with $n(q)\leq m(q)$ for any  $q\in \dn$. 

Since $\dn$ is an infinite set and the number of possible pairs $(n(q),m(q))$ is finite, there exist two different positive numbers $p,q$, $p<q$ such that $n(p)=n(q)$, $m(p)=m(q)$.

\begin{lemma}
\label{L:calculations}
Let $\varphi$, $P$, $(\sigma ,\rho)$, $p<q$ be as above. Put $n:=n(p)=n(q)$, $m:=m(p)=m(q)$.  Then $n=m$ and 
$$v_{\sigma,\rho}(P)=m(\sigma+\rho).$$
In particular, $n$ and $m$ don't depend on the choice of $p<q$. 
\end{lemma}

\begin{proof}
Assume $\theta_{0}=v_{\sigma,\rho}(P)$ (note that $\theta_{0}>0$ since $P\notin K$), and set 
$$ \theta_{1}=v_{\sigma,\rho}(ad\varphi^{q}(\partial)P),\quad  \ldots , \quad \theta_{m}=v_{\sigma,\rho}([ad\varphi^{q}(\partial)]^{m}P).
$$
From our assumption on $m$ and theorem \ref{T:Dixmier2.7} it follows that the equations 
$$
\theta_{i}=\theta_{i-1}+v_{\sigma,\rho}(\varphi^{q}(\partial))-\sigma-\rho
$$
 hold for $1\leq i\leq m$. However, we have (cf. formulae \eqref{E:cases})
\begin{equation}\label{E:vii}
	\theta_{m}=mk(k+\epsilon l)^{q}+\theta_{0}-m(\sigma+\rho).
\end{equation}
On another side, set 
$$\eta_{0}=\theta_{0}=v_{\sigma,\rho}(P), \quad \eta_{1}=v_{\sigma,\rho}(ad\varphi^{p}(\partial)P), \quad \ldots ,\quad \eta_{m}=v_{\sigma,\rho}([ad\varphi^{p}(\partial)]^{m}P).
$$
Then we have 
 \begin{equation}\label{E:viii}
 	\eta_{m}=mk(k+\epsilon l)^{p}+\eta_{0}-m(\sigma+\rho)
 \end{equation}
Since $f_{\sigma,\rho}([ad\varphi^{p}(\partial)]^{m}(P))$ and $f_{\sigma,\rho}([ad\varphi^{q}(\partial)]^{m}(P))$ are both monomials according to our assumptions, then condition in Lemma \ref{L:9} is satisfied. Applying $\varphi$ $(q-p)$ times and using \eqref{E:vi} (taking into account that $(\rho k+\sigma l) v_{\epsilon ,1}(P)= (k+\epsilon l) v_{\sigma ,\rho }(P)$), we get
\begin{equation}\label{E:ix}
	\theta_{m}=(k+\epsilon l)^{q-p}\eta_{m}.
\end{equation}
Comparing with \eqref{E:vii}, \eqref{E:viii}, \eqref{E:ix}, we get
\begin{equation}\label{E:x}
	\theta_{0}=m(\sigma+\rho).
\end{equation}
We can rewrite this procedure for $n$ instead of $m$ in the same way, but since the result \eqref{E:x} has no relation with $p,q$, we immediately get $n=m$, which means, the first time $f_{\sigma,\rho}([ad\varphi^{p}(\partial)]^{n}P)$ becomes a monomial, it almost commutes with $f_{\sigma,\rho}(\varphi(\partial))$, i.e. is proportional to it. In other words, $[ad\varphi^{q}(\partial)]^{m-1}(P)$ is not of monomial type associated to $(\sigma,\rho)$.
\end{proof}

\subsection{Step 3}  

Again, suppose $\varphi$ is an endomorphism of subrectangular type with a fixed point $P\notin K$ (i.e $\varphi(P)=P$), and suppose $f_{\sigma,\rho}(\varphi(\partial))=c_1x^{l}y^{k}, f_{\sigma,\rho}(\varphi(x))=c_2x^{l'}y^{k'}$, with $(l',k')=\epsilon (l,k)$, $v_{\sigma,\rho}(P)=m(\sigma+\rho)$ as in the previous step. We are now going to show it's impossible.

First we should notice that $f_{\sigma,\rho}(P)$ can't be a monomial. If not, since $\varphi(P)=P$, we have $v_{\sigma,\rho}(\varphi(P))=v_{\sigma,\rho}(P)$, on the other hand, by lemma \ref{L:9} we know that $v_{\sigma,\rho}(\varphi(P))=(\rho k+\sigma l) v_{\epsilon ,1}(P)=(k+\epsilon l)v_{\sigma,\rho}(P)$, it's impossible as $v_{\sigma,\rho}(P)>0$.

Now assume $f_{\sigma,\rho}(P)=c_{1}x^{n_{1}}y^{m_{1}}+c_{2}x^{n_{2}}y^{m_{2}}+\cdots$, $c_i\neq 0$ (here $n_{i},m_{i}$ has no relation with $n,m$). According to the definition of weight degree $v_{\sigma,\rho}$ and by lemma \ref{L:calculations}, we know that for any integer $i$ 
\begin{equation}\label{E:xi}
	\sigma n_{i}+\rho  m_{i}=m(\sigma+\rho). 
\end{equation}

Let $p,q$ be as in Lemma \ref{L:calculations}. We have calculate  before (see \eqref{E:cases}):
$$\begin{cases}
f_{\sigma,\rho}(\varphi^{q}(\partial))=\tilde{c}_1x^{l(k+\epsilon l)^{q-1}}y^{k(k+\epsilon l)^{q-1}}
\\
f_{\sigma,\rho}(\varphi^{q}(x))=\tilde{c}_2x^{l\epsilon(k+\epsilon l)^{q-1}}y^{k\epsilon(k+\epsilon l)^{q-1}}
\end{cases}$$
Denote $u=l(k+\epsilon l)^{q-1}, v=k(k+\epsilon l)^{q-1}$ for convenience. According to the definition of $m$, we know $m+1$ is the first number such that $\{ad (f_{\sigma,\rho}(\varphi^{q}(\partial))) \}^{m+1}(f_{\sigma,\rho}(P))=0$.

Consider $$\{x^{u}y^{v},x^{n_{i}}y^{m_{i}}  \}=(vn_{i}-um_{i})x^{n_{i}+u-1}y^{m_{i}+v-1} $$
 and for any  $r\in \dn$ $\exists \lambda=\lambda(r,i) \in K$, such that 
 $$ \{adx^{u}y^{v}  \}^{r}(x^{n_{i}}y^{m_{i}})=\lambda x^{n_{i}+r(u-1)}y^{m_{i}+r(v-1)}$$
  Since $f_{\sigma,\rho}([ad\varphi^{q}(\partial)]^{m}(P))$ is a monomial, the homogeneous terms in $P$ must be eliminated by taking multiple Poisson bracket until there is only one term left.

Notice that if $i\neq j$, then   
$$\lambda(r,i) x^{n_{i}+r(u-1)}y^{m_{i}+r(v-1)}= \lambda(r,j) x^{n_{j}+r(u-1)}y^{m_{j}+r(v-1)}$$
only if $\lambda(r,i)=\lambda(r,j)=0$. 
This means the terms will be eliminated only if $\lambda(r,i)=0$ for some $r$.

Now suppose $\xi =\xi (i)\le m$ is the last integer that $\lambda=\lambda(\xi,i )$ is not equal to 0. This means  $\{x^{u}y^{v},x^{n_{i}+\xi(u-1)}y^{m_{i}+\xi(v-1)}\}=0$. Thus, we have 
	$$v(n_{i}+\xi(u-1))-u(m_{i}+\xi(v-1))=0. $$

Simplifying the equation, we have $v n_{i}-um_{i}-v\xi+u\xi=0$. Since $u=l(k+\epsilon l)^{q-1}, v=k(k+\epsilon l)^{q-1}$, we can eliminate the common factor $(k+\epsilon l)^{q-1}$, thus we have 
	$$k n_{i}-l m_{i}-k\xi+l\xi=0. $$

Eliminate $m_{i}$ by \eqref{E:xi}, we have
	$$k n_{i}-l(\frac{m(\sigma+\rho)-\sigma n_{i}}{\rho}) -k\xi+l\xi=0. $$

From this we can get a representation of $n_{i}$: 
\begin{equation}\label{xii}
	n_{i}=\frac{lm(\sigma+\rho)+\xi\rho(k-l)}{k\rho+l\sigma}
\end{equation}
If $\xi =m$, we can solve the equation and get $n_{i}=m_{i}=m$. If  $\xi =m-1$, then $n_{i}$ can not be an integer since $ \rho(k-l)<k\rho+l\sigma$. This means that $c_m x^m y^m= f_{\sigma ,\rho}([ad\varphi^{q}(\partial)]^{m}(P))$,  where $c_m\neq 0$, is the last monomial surviving after applying the Poisson bracket, because otherwise all $\xi (i)<m-1$, and then $\{ad (f_{\sigma,\rho}(\varphi^{q}(\partial))) \}^{m-1}(f_{\sigma,\rho}(P))=0$, a contradiction with definition of $m$. 
On the other hand, this also means there are no monomials $x^{m_{j}}\partial^{n_{j}}$ eliminated after applying the Poisson bracket  $(m-1)$-th time, and this means $[ad\varphi^{q}(\partial)]^{m-1}(P)$ is of monomial type associated to $(\sigma,\rho)$. This contradicts to the result we get in lemma \ref{L:calculations}, i.e. that $m=n$, thus $f_{\sigma,\rho}(P)$ must be a monomial - a contradiction.

\begin{rem}
There is an interesting question  by A. Ya. Kanel-Belov, whether  one can use the degree estimates from \cite{ML1}, \cite{ML-Y}, \cite{BK-Y} to give an alternative proof of the absence of the fixed points  in  theorem \ref{T:fixed}. It seems to be non-trivial. 
\end{rem}

\noindent J. Guo,  School of Mathematical Sciences, Peking University and Sino-Russian Mathematics Center,  Beijing, China and 
Lomonosov Moscow State  University, faculty of mechanics and mathematics, department of differential geometry and applications, Leninskie gory, GSP, Moscow, \nopagebreak 119899,
\\ 
\noindent\ e-mail:
$123281697@qq.com$

\vspace{0.5cm}

\noindent A. Zheglov, Lomonosov Moscow State  University, faculty
of mechanics and mathematics, department of differential geometry
and applications, Leninskie gory, GSP, Moscow, \nopagebreak 119899,
Russia
\\ \noindent e-mail
 $azheglov@mech.math.msu.su$


\begin{thebibliography}{99}


\bibitem{BZ} Burban I., Zheglov A., {\em Fourier-Mukai transform on Weierstrass cubics and commuting differential operators}, International Journal of Mathematics. - 2018. - P. 1850064

\bibitem{BC} J.~Burchnall, T.~Chaundy, \emph{Commutative ordinary differential operators},
 Proc.~London Math.~Soc.~\textbf{21}  (1923) 420--440.

\bibitem{MirDav} Davletshina V.N., Mironov A.E., {\it On commuting ordinary differential operators with polynomial coefficients corresponding to spectral curves of genus two}, Bull. Korean Math. Soc. 54 (2017), 1669–1675

\bibitem{Dixmier} J.~Dixmier, \textit{Sur les alg\`ebres de Weyl},  Bull.~Soc.~Math.~France \textbf{96} (1968) 209--242.

\bibitem{GGV} Jorge A. Guccione, Juan J. Guccione, C. Valqui, {\it The Dixmier conjecture and the shape of possible counterexamples}, J. of Algebra, 399 (2014), 581--633.

\bibitem{G} P.G. Grinevich, {\it Rational solutions for the equation of commutation of differential operators}, Functional Anal. Appl., {\bf 16}:1 (1982), 15--19.

\bibitem{Grun} F.~Gr\"unbaum, \textit{Commuting pairs of linear ordinary differential operators of orders four and six}, Phys.~D \textbf{31} (1988), 424--433.

\bibitem{Gundareva} A. Gundareva, {\em On commuting elements in the first Weyl algebra over $\dq$},  Siberian Mathematical Journal. — 2022.

\bibitem{Joseph} A. Joseph, \textit{The Weyl algebra-semisimple and nilpotent elements}, American Journal of Mathematics \textbf{97} (1975), 597-615. MR0379615(52:520)


\bibitem{BK} A. Ya. Kanel-Belov, M. L. Kontsevich,
{\it The Jacobian conjecture is stably equivalent to the Dixmier conjecture},
Mosc. Math. J., {\bf 7}:2 (2007),  209-218.

\bibitem{BK-Y} A. Ya. Kanel-Belov, Jie-Tai Yu, {\it On the lifting of the Nagata automorphism}, Selecta Mathematica, 17:4 (2011), 935--945

\bibitem{KN} I.~Krichever, S.~Novikov, \textit{Holomorphic bundles over algebraic curves and nonlinear equations}, Russian Math.~Surveys, {\bf 35}:6 (1980), 47--68.

\bibitem{ML1} L.Makar-Limanov, {\it The automorphisms of the free algebra with two generators},
Funkcional. Anal. 4 (1970) 107-108.

\bibitem{ML} L.G. Makar-Limanov, {\it Centralizers of Rank One in the First Weyl Algebra}, SIGMA, 17
(2021), 052, 13 pp.

\bibitem{ML-Y}  L. Makar-Limanov and Jie-Tai Yu, {\it Degree estimate for subalgebras generated
by two elements}, J. Eur. Math. Soc. 10 (2008) 533-541.

\bibitem{M1} A.E. Mironov, {\it Self-adjoint commuting ordinary differential
operators}, Invent math, {\bf 197}: 2 (2014), 417--431
DOI 10.1007/s00222-013-0486-8.

\bibitem{M2} A.E. Mironov, {\it Periodic and rapid decay rank two self-adjoint commuting differential operators.}, Amer. Math. Soc. Transl. Ser. 2, V. 234, 2014, 309--322.

\bibitem{MZh} A.~E. Mironov, A.~B. Zheglov, \textit{Commuting ordinary differential operators with polynomial coefficients and automorphisms of the first Weyl algebra}, Int. Math. Res. Not. IMRN, {\bf 10}, 2974--2993 (2016).

\bibitem{SMZh} A.~E. Mironov, B. Saparbaeva, A.~B. Zheglov, {\it Commuting krichever–novikov differential operators with polynomial coefficients}, Siberian Mathematical Journal. — 2016. — Vol. 57. — P. 819–10. 

\bibitem{Mokh} O.I. Mokhov, \textit{Commuting differential operators of rank 3 and nonlinear differential equations},  Mathematics of the USSR-Izvestiya {\bf 35}:3 (1990), 629--655.


\bibitem{Mokh1} O.I. Mokhov, {\it On commutative subalgebras of the Weyl algebra related to commuting operators of arbitrary rank and genus}, Mathematical Notes, {\bf 94}:2 (2013), 298--300.


\bibitem{Mokh2} O.I. Mokhov, {\it Commuting ordinary differential operators of arbitrary genus and arbitrary rank with polynomial coefficients},  Amer. Math. Soc. Transl. Ser. 2, V. 234, 2014, 309--322.

\bibitem{Vardan34} Oganesyan V.S., {\it Commuting differential operators of rank 2 with polynomial coefficients}, Funct. Anal. Appl. 50 (2016), 54–61

\bibitem{PZ} Pogorelov D. A., Zheglov A. B., {\it An algorithm for construction of commuting ordinary
differential operators by geometric data}, Lobachevskii Journal of Mathematics. 2017.
Vol. 38, no. 6. P. 10751092.

\bibitem{Previato2019} E. Previato, S.L. Rueda, M.-A. Zurro, {\it Commuting Ordinary Differential Operators and the Dixmier Test},  SIGMA 15 (2019), 101, https://doi.org/10.3842/SIGMA.2019.101 


\bibitem{PW} E. Previato, G. Wilson, \textit{Differential operators and rank $2$
 bundles over elliptic curves}, Compositio Math. \textbf{81}:1 (1992), 107-119.

\bibitem{Schur} I.~Schur, \emph{\"Uber vertauschbare lineare Differentialausdr\"ucke},
Sitzungsber.~Berl.~Math.~Ges.~\textbf{4}, 2--8 (1905).

\bibitem{Ts1} Y. Tsuchimoto, {\it Preliminaries on Dixmier Conjecture}, Mem. Fac. Sci. Kochi Univ., Ser. A
Math. 24 (2003), 43-59.

\bibitem{Ts2} Y. Tsuchimoto, {\it Endomorphisms of Weyl algebra and p-curvatures}, Osaka J. Math. 42 (2005), No. 2, 435-452.

\bibitem{Zheglov_book} Zheglov A. B., {\em Algebra, geometry and analysis of commuting ordinary differential operators}, Publ. house of the Board of trustees of the Faculty of mechanics and mathematics, Moscow state univ., 2020. - 217,  ISBN 978-5-9500628-4-1\\ 
can be found e.g. at 
https://www.researchgate.net/publication/\\
340952902AlgebraGeometryandAnalysisofCommutingOrdinaryDifferentialOperators 





\end{thebibliography}
\end{document}